\documentclass[12pt,leqno]{amsart}
\usepackage{amsmath,amsfonts,amssymb,amsthm}
\usepackage{graphicx}
\graphicspath{ }
\usepackage{wrapfig}
\usepackage{accents}
\usepackage{caption}
\usepackage{subcaption}
\usepackage{calligra}

\numberwithin{figure}{section}

\theoremstyle{plain}
\newtheorem{thm}{Theorem}[section]
\newtheorem{lem}[thm]{Lemma}
\newtheorem{prop}[thm]{Proposition}
\newtheorem{cor}{Corollary}[thm]
\theoremstyle{definition}

\theoremstyle{remark}

\usepackage{mathtools}
\title{Isometry theorem of Cartan-Hadamard manifold}
\author[A. A. Shaikh, P. Mandal, C. K. Mondal and P. R. Ghosh]{Absos Ali Shaikh$^{*1}$, Prosenjit Mandal$^{2}$, Chandan Kumar Mondal$^{3}$ and Pinaki Ranjan Ghosh$^{4}$}
\address{\noindent\newline  Department of Mathematics,\newline University of
Burdwan, Golapbag,\newline Burdwan-713104,\newline West Bengal, India}
\email{$^1$aask2003@yahoo.co.in, aashaikh@math.buruniv.ac.in}
\email{$^2$prosenjitmandal235@gmail.com}
\email{$^3$chan.alge@gmail.com}
\email{$^4$mailtopinaki94@gmail.com}
\begin{document}

\begin{abstract}
Cartan-Hadamard manifold is a simply connected Riemannian manifold with non-positive sectional curvature. In this article, we have proved that a Cartan-Hadamard manifold satisfying steady gradient Ricci soliton with the integral condition of potential function is isometric to the Euclidean space. Next we have proved a compactness theorem for gradient shrinking Ricci soliton satisfying some scalar curvature condition. Finally, we have showed that a gradient expanding Ricci soliton with linear volume growth and positive potential function is an Einstein manifold.
\end{abstract}
\noindent\footnotetext{
$\mathbf{2020}$\hspace{5pt}Mathematics\; Subject\; Classification: 53C20; 53C21. }
\maketitle
\section*{Introduction}
 The notion of Ricci soltion is developed when Hamiltion \cite{HA82} introduced the concept of Ricci flow. A complete Riemannian manifold $(M,g)$ of dimension $n\geq 2$, with Riemannian metric $g$ is called a Ricci soliton if it admits a smooth vector field $X$ satisfying the following equation:
\begin{equation}\label{r7}
Ric+\frac{1}{2}\pounds_Xg=\lambda g,
\end{equation}
where $\lambda$ is a constant and $\pounds$ denotes the Lie derivative. The smooth vector field $X$ is called potential vector field.  If the vector field $X$ is zero or Killing, then the Ricci soliton becomes Einstein. Throughout the paper by $M$ we mean an $n$-dimensional, $n\geq 2$, complete Riemannian manifold endowed with Riemannian metric $g$. If $X=\nabla f$, where $\nabla$ is the gradient operator and $f\in C^\infty(M)$, the ring of smooth functions in $M$, then the Ricci soliton is called gradient Ricci soliton and (\ref{r7}) reduces to the form
\begin{equation}\label{r1}
\nabla^2f+Ric=\lambda g,
\end{equation}
where $\nabla^2f$ is the Hessian of $f$ and the function $f$ is called potential function. The Ricci soliton $(M,g)$ is called shrinking, steady and expanding according as $\lambda>0$, $\lambda=0$ and $\lambda<0$, respectively. By scaling the metric we can take $\lambda=-\frac{1}{2},0$ or $\frac{1}{2}$ for expanding, steady and shrinking respectively. By adding some constant to $f$, gradient Ricci soliton implies \cite{HA82}
\begin{equation}\label{eq4}
R+|\nabla f|^2=\lambda f.
\end{equation}
 Recently, the study of Ricci solitons is an important theme of research to mathematicians as well as physicists. Each type of Ricci soliton has a significant impact on the topology of the manifold. Munteanu and Wang \cite{MW17} proved a compactness theorem for the $n$-dimensional gradient shrinking Ricci soliton with non-negative sectional curvature and positive Ricci curvature.  Perelman \cite{PE03} proved that a compact Ricci soliton is always gradient Ricci soliton. Fern\'andez-L\'opez and Garc\'ia-R\'io \cite{FG08} proved that  shrinking Ricci soliton with bounded vector field must be compact. Li and Zhou \cite{LZ19} proved that compact Ricci soliton with vanishing Weyl conformal curvature tensor must be Einstein.\\
\par Cartan-Hadamard theorem states that the universal cover of an $n$-dimensional complete Riemannian manifold with non-positive curvature is diffeomorphic to the $n$-dimensional Euclidean space $\mathbb{R}^n$. In particular, if the manifold is simply connected then itself is diffeomorphic to $\mathbb{R}^n$. More precisely, at any point $p\in M$, the exponential mapping  $exp_p : T_pM\rightarrow$ M is the diffeomorphism, where $T_pM$ is the tangent space at $p$.  From this theorem K. Shiga \cite{SH84} developed the notion of Cartan-Hadamard manifold. Cartan-Hadamard manifold, in general, may not be isometric to the Euclidean space. Therefore, it is quite natural to ask under which condition Cartan-Hadamard manifold is isometric to the Euclidean space. In the first section we have proved that Cartan-Hadamard manifold satisfying steady gradient Ricci soliton with potential function possessing some integral condition is isometric to the Euclidean space. In the second section, we have showed that gradient shrinking Ricci soliton with bounded potential function and quadratic volume satisfying some scalar curvature condition is compact. In the last section, we have derived a condition for which an expanding gradient Ricci soliton becomes an Einstein manifold.
\section{Steady Ricci solition in Cartan-Hadamard manifold}

\begin{thm}
Let $(M,g)$ be a Cartan-Hadamard manifold satisfying gradient steady Ricci soliton with the potential function $f\leq 0$ satisfying the condition
\begin{equation}\label{eq6}
 -\int_{M-B(p,r)}d(x,p)^{-2}f<\infty,
 \end{equation}
 where $B(p,r)$ is the ball with center at $p$ and radius $r>0$. Then $M$ is isometric to $\mathbb{R}^n$.
\end{thm}
To prove the Theorem $1.1$, we need the following Lemma which will state and prove at first.
\begin{lem}\label{lem1}
Let $(M,g)$ be a complete gradient steady Ricci soliton with the potential function $f\leq 0$ satisfying (\ref{eq6}).
 Then the scalar curvature vanishes in $M$.
\end{lem}
\begin{proof}
For steady Ricci soliton, from the canonical form of Ricci soliton we get
\begin{equation}\label{eq1}
R+\Delta f=0, 
\end{equation}
\begin{equation}\label{eq5}
R+|\nabla f|^2=A,
\end{equation}
for some constant $A\geq 0$. Steady gradient Ricci soliton imply that $R\geq 0$, see \cite{CH09}. Let $u=-f$ and then from (\ref{eq1}), we obtain 
\begin{equation}\label{eq2}
R= \Delta u.
\end{equation}
 Now, we consider the cut-off function, introduced in \cite{CC96}, $\varphi_r\in C^\infty_0(B(p,2r))$ for $r>0$ such that
\[ \begin{cases} 
	  0\leq \varphi_r\leq 1 &\text{ in }B(p,2r)\\
      \varphi_r=1  & \text{ in }B(p,r) \\
      |\nabla \varphi_r|^2\leq\frac{C}{r^2}& \text{ in }B(p,2r) \\
      \Delta \varphi_r\leq \frac{C}{r^2} &  \text{ in }B(p,2r),
   \end{cases}
\]
where $C>0$ is a constant.
Then for $r\rightarrow\infty$, we have $\Delta \varphi_r\rightarrow 0$ as $\Delta \varphi_r\leq \frac{C}{r^2}$.
Now using integration by parts, we have 
\begin{eqnarray}\label{eq3}
 \nonumber\int_{B(p,2r)} \phi_r \Delta u =  \int_{B(p,2r)} \Delta \phi_ru&=&\int_{B(p,2r)-B(p,r)}\Delta \phi_ru\\
&\leq & C\int_{B(p,2r)-B(p,r)}\frac{1}{r^2}u.
\end{eqnarray}
In the last inequality, we have used the property of $\varphi_r$.
Thus (\ref{eq2}), (\ref{eq3}) and our assumption together imply that
\begin{equation}
 \int_{B(p,2r)}\varphi_rR=\int_{B(p,2r)} \phi_r \Delta u\leq C\int_{B(p,2r)-B(p,r)}\frac{1}{r^2}u\rightarrow 0,
\end{equation}
as $r\rightarrow\infty$.
But $0\leq \varphi_r\leq 1$ in $B(p,2r)$ and $R\leq 0$, hence it follows that for all $r>0$
$$\int_{B(p,2r)}\varphi_rR\leq 0.$$
Therefore, we get
$$\int_{B(p,2r)}\varphi_rR= 0\quad \text{ for all }r>0.$$
Hence, $R$ vanishes in $M$. 
\end{proof}

\begin{proof}[\textbf{Proof of Theorem 1.1.}]
From the Lemma \ref{lem1} we see that $R=0$ in $M$. Therefore, (\ref{eq5}) implies that $|\nabla f|^2=A$. Also from (\ref{eq1}), we get $\Delta f=0$. The Bochner formula \cite{AU13} for the Riemannian manifold is 
$$\frac{1}{2}\Delta|\nabla f|^2=|\nabla^2 f|^2+g(\nabla f,\nabla\Delta f)+Ric(\nabla f,\nabla f).$$
Now placing $\Delta f=0$, the above equation reduces to
\begin{equation}\label{r1.1}
\frac{1}{2}\Delta|\nabla f|^2=|\nabla^2 f|^2+Ric(\nabla f,\nabla f).
\end{equation}
Again, $\Delta|\nabla f|^2=0$. It follows from the above equation that $Ric(\nabla f,\nabla f)=0$. Since, $M$ has non-positive sectional curvature $K$ and 
$$Ric_p(\nabla f,\nabla f)=\sum_{i=1}^{n-1}K_p(\nabla f,e_i),$$
where $\{\nabla f,e_1,\cdots,e_{n-1}\}$ is the orthonormal basis for $T_pM$, the sectional curvature vanishes everywhere. Therefore, we conclude that $M$ is isometric to $\mathbb{R}^n$.
\end{proof}
\begin{lem}\cite{SA96}\label{lm1}
Let $f$ be a smooth function in a complete Riemannian manifold with non-negative Ricci curvature. If $|\nabla f|$ is constant, then $f$ is an affine function.
\end{lem}
\begin{thm}\cite[Theorem 1]{IN82}\label{th5}
If a complete Riemannian manifold $(M,g)$ admits a non-constant smooth affine function, then $M$ is isometric to $N\times\mathbb{R}$, for a totally geodesic submanifold $N$ of $M$.
\end{thm}
\begin{thm}
Let $(M,g)$ be a complete gradient steady Ricci soliton with non-negative Ricci curvature and the potential function $f\leq 0$ satisfying (\ref{eq6}). Then $M$ is isometric to $N\times\mathbb{R}$, for a totally geodesic submanifold $N$ of $M$.
\end{thm}
\begin{proof}
From Lemma \ref{lem1} we get that the scalar curvature vanishes in $M$. Thus (\ref{eq5}) implies that $|\nabla f|^2=A$ in $M$. Since $M$ has non-negative Ricci curvature, It follows from Lemma \ref{lm1} that $f$ is a affine function. Therefore, from Theorem \ref{th5}, we conclude the result.
\end{proof}
\section{Compactness theorem for shrinking Ricci soliton}
\begin{thm}
Let $(M, g)$ be an n-dimensional complete gradient shrinking Ricci soliton with the bounded potential function $f$ and non-negative sectional curvature. If $M$ has quadratic volume growth and there are constants $r_0>0$ and $C_1>0$ such that the scalar curvature $R\leq \frac{C_1}{r^2}$ for all $r\geq r_0$, then $M$ is compact.
\end{thm}

To prove the Theorem $2.1$, we need the following results:

\begin{thm}\cite{GW74}\label{G1}
Let $M$ be a non-compact complete Riemannian manifold with non-negative sectional curvature. If $f\in C^\infty(M)$ is a non-negative subharmonic function, then
$$\int_M f=+\infty.$$
\end{thm}

\begin{lem}\label{lem2}
Let $(M,g)$ be a complete gradient shrinking Ricci soliton with potential function $f$. If the Ricci curvature is non-negative, then the potential function $f$ satisfies the following inequality:
\begin{equation}
\int_{B(p,r)}f \leq  \frac{C}{r^2}\int_{B(p,2r)}f^2+\frac{n}{2}Vol(B(p,r)),
\end{equation}
for all $p\in M$ and $r>0$, where $C>0$ is a constant.
\end{lem}
\begin{proof}
From the gradient shrinking Ricci soliton, we get
\begin{equation}\label{s2}
R+|\nabla f|^2=f,
\end{equation}
and then by taking trace, we get
\begin{equation}\label{s1}
R+\Delta f=\frac{n}{2}.
\end{equation}
Now, for the gradient shrinking Ricci soliton, the scalar curvature satisfies the inequality \cite{CZ10} $$R\leq \frac{n}{2}.$$
It follows from (\ref{s1}), that $\Delta f\geq 0$, i.e., $f$ is subharmonic. Again $R\geq 0$ for gradient shrinking Ricci soliton \cite{CZ10} implies that $f\geq 0$, for some constant $K>0$.
Now we take the same cut-off function $\varphi_r$ defined in Lemma \ref{lem1}. Then we calculate
\begin{equation*}
0\leq  \int_{B(p,2r)}\varphi_r^2f\Delta f=-\int_{B(p,2r)}\varphi_r^2|\nabla f|^2-2\int_{B(p,2r)}\varphi_rfg\Big(\nabla\varphi_r,\nabla f\Big).
\end{equation*}
Then we have
\begin{eqnarray*}
&&\int_{B(p,2r)}\varphi_r^2|\nabla f|^2 \leq  -2\int_{B(p,2r)}\varphi_rfg\Big(\nabla\varphi_r,\nabla f\Big)\\
&&\leq 2\Big(\int_{B(p,2r)}\varphi_r^2|\nabla f|^2 \Big)^{1/2}\Big(\int_{B(p,2r)}f^2|\nabla \varphi_r|^2 \Big)^{1/2}.
\end{eqnarray*}
Therefore, we obtain
\begin{eqnarray}\label{in1}
\nonumber\int_{B(p,r)}|\nabla f|^2 &\leq & \int_{B(p,2r)}\varphi_r^2|\nabla f|^2 \leq 4\int_{B(p,2r)}f^2|\nabla \varphi_r|^2\\
&\leq & \frac{C}{r^2}\int_{B(p,2r)}f^2.
\end{eqnarray}
Now using (\ref{s2}), we obtain
\begin{equation}\label{in2}
\int_{B(p,2r)}\Big(f-R\Big)\leq \frac{C}{r^2}\int_{B(p,2r)}f^2,
\end{equation}
which implies that
\begin{eqnarray}
\label{s3}\int_{B(p,r)}f &\leq & \frac{C}{r^2}\int_{B(p,2r)}f^2+\int_{B(p,r)}R\\
&\leq & \frac{C}{r^2}\int_{B(p,2r)}f^2+\frac{n}{2}Vol(B(p,r)).
\end{eqnarray}
\end{proof}

\begin{proof}[\textbf{Proof of Theorem 2.1.}]
By placing $R\leq \frac{C_1}{r^2}$ in (\ref{s3}), the inequality (\ref{s3}) reduces to the form
\begin{equation}
\int_{B(p,r)}f \leq  \frac{C}{r^2}\int_{B(p,2r)}f^2+\int_{B(p,r)}R,
\end{equation}
for $r\geq r_0$. Now the manifold $M$ has quadratic volume growth, i.e., $Vol(B(p,r))\leq Kr^2$, for some constant $K>0$ and also there is a constant $C'>0$ such that $f^2\leq C'$. Therefore, the above inequality implies that
\begin{eqnarray*}
\int_{B(p,r)}f &\leq & \frac{C'C}{r^2}Vol(B(p,2r))+\frac{C_1}{r^2}Vol(B(p,r))\\
&\leq & K',
\end{eqnarray*}
for all $r\geq r_0$ and for some constant $K'>0$. Taking limit as $r\rightarrow\infty$ we obtain
$$\int_M f\leq K'.$$
Therefore, from Theorem \ref{G1} we see that $M$ must be compact.
\end{proof}
\begin{cor}
Let $(M,g)$ be a gradient shrinking Ricci soliton with bounded potential function $f$. If $M$ has linear volume growth, then $M$ is an Einstein manifold.
\end{cor}
\begin{proof}
Since scalar curvature vanishes in $M$, from (\ref{s3}), we get
\begin{equation*}
\int_{B(p,r)}|\nabla f|^2\leq \frac{C}{r^2}\int_{B(p,2r)}f^2.
\end{equation*}
Since $f \leq k$ for some $k>0$, it follows that
\begin{equation*}
\int_{B(p,r)}|\nabla f|^2 \leq \frac{Ck^2}{r^2}Vol(B(p,2r)).
\end{equation*}
Using the linear volume growth, i.e., $Vol(B(p,r))\leq c'r$, for some constant $c'>0$, we have
\begin{equation*}
\int_{B(p,r)}|\nabla f|^2 \leq \frac{Ck^2}{r^2}(2rc')=\frac{2Cc'k^2}{r}.
\end{equation*}
Taking limit $r\rightarrow\infty$, we obtain
\begin{equation*}
\int_{M}|\nabla f|^2 =0,
\end{equation*}
which follows that the potential function $f$ is constant. Consequently $M$ is an Einstein manifold.  
\end{proof}
\begin{cor}
Suppose $(M,g)$ is a gradient shrinking Ricci soliton with bounded potential function $f$. If $M$ has quadratic volume growth and the scalar curvature vanishes in $M$, then $f$ has finite integral in $M$.
\end{cor}
\begin{proof}
The proof easily follows from (\ref{s3}) and hence we omit.
\end{proof}
\begin{prop}
The potential function of a gradient shrinking $($resp., expanding$)$ Ricci soliton can not be strictly superharmonic $($resp., subharmonic$)$.
\end{prop}
\begin{proof}
 Now from fundamental equation of Ricci soliton we have
\begin{equation*}
Ric+\nabla^2f=\frac{1}{2} g.
\end{equation*}
Taking trace of the above equation, we get
\begin{equation}\label{P11}
R+\Delta f=n \frac{1}{2}.
\end{equation} 
Let us consider the potential function $f$ as superharmonic, i.e., $\Delta f\leq 0$. Hence from (\ref{P11}) we can write
\begin{equation*}
R\geq \frac{n}{2},
\end{equation*}
which contradicts that $R\leq \frac{n}{2}$ for gradient shrinking Ricci solitons. So $f$ can not be superharmonic.

Similarly we can proof the case of expanding gradient Ricci soliton.
\end{proof}
\section{Expanding Ricci soliton and Einstein manifold}
\begin{thm}
Let $(M, g)$ be an n-dimensional complete gradient expanding Ricci soliton with the potential function $f\geq k $ for some constant $k>0$ and $M$ be of linear volume growth. Then M is an Einstein manifold.
\end{thm} 
\begin{proof}
Since $M$ is gradient expanding Ricci soliton, the scalar curvature $R\geq -\frac{n}{2}$. Hence, from equation (\ref{eq4}) we get  $\Delta f $ $\leq 0$, i.e, $M$ is superharmonic. Now
\begin{equation}
 \Big(\frac{1}{f}\Big)_i = -\Big(\frac{1}{f^2}\Big)f_i
\end{equation}

and
\begin{equation}
 \Big(\frac{1}{f}\Big)_{ii} = \Big(\frac{2}{f^3}\Big)f^2_i -\Big(\frac{1}{f^2}\Big)f_{ii}.
\end{equation}
Therefore
\begin{equation}
\Delta \Big(\frac{1}{f}\Big)=\Big(\frac{2}{f^3}\Big) |\nabla f|^2 -\frac{\Delta f}{f^2}.
\end{equation}
Since $\Delta f\leq 0$, it implies that $\Delta (\frac{1}{f})\geq 0$. Hence $\frac{1}{f}$ is subharmonic. Again since $f$ is positive, $\frac{1}{f}$ is also positive. The manifold $M$ has linear volume growth, $Vol(B(p,r))\leq C_1r$ for some constant $C_1>0$.
Now for non-negative subharmonic function we obtain from the equation (\ref{in1}) that
\begin{eqnarray*}
\int_{B(p,r)}|\nabla \frac{1}{f}|^2&&\leq \Big(\frac{C}{r^2}\Big)\int_{B(p,2r)}\Big(\frac{1}{f^2}\Big)\\
 && \leq \Big(\frac{C}{r^2k^2}\Big) Vol(B(p,2r)) \\
 &&\leq \Big(\frac{C}{r^2k^2}\Big)C_1 2r\\
&&\leq \frac{2CC_1}{rk^2}.  
\end{eqnarray*}
Taking limit as $r\rightarrow \infty$, we have 
\begin{equation}
 \int_{M}|\nabla \frac{1}{f}|^2=0,
\end{equation}
which follows that the function $\frac{1}{f}$ is constant. Consequently the potential function $f$ is constant. It follows from (2) that M is an Einstein manifold.
\end{proof}


\section{acknowledgment}
 The second author gratefully acknowledges to the
 CSIR(File No.:09/025(0282)/2019-EMR-I), Govt. of India for financial assistance. The forth author greatly acknowledges to The University Grants Commission, Government of India for the award of Junior Research Fellow.

\end{document}